\documentclass{conm-p-l}

\usepackage{amssymb}

\usepackage{graphicx}

\usepackage[cmtip,all]{xy}

\usepackage{}

\newtheorem{theorem}{Theorem}[section]
\newtheorem{lemma}[theorem]{Lemma}
\newtheorem{rem}[theorem]{Remark}
\newtheorem{corollary}[theorem]{Corollary}
\newtheorem*{assume*}{Assumptions}
\theoremstyle{definition}
\newtheorem{definition}[theorem]{Definition}

\theoremstyle{remark}

\numberwithin{equation}{section}

\begin{document}

\title[The Dixmier-Moeglin equivalence]{The Dixmier-Moeglin equivalence for extensions of scalars and Ore extensions}







\title[The Dixmier-Moeglin equivalence]{The Dixmier-Moeglin equivalence for extensions of scalars and Ore extensions}

\author{Jason P. Bell}
\author{Kaiyu Wu}
\author{Shelley Wu}
\thanks{The research of the first-named author  was supported by NSERC Grant 326532-2011; the second- and third-named authors were supported by an NSERC USRA award.}
\address{ Department of Mathematics, University of Waterloo, 
Waterloo, ON, N2L  3G1, CANADA}
\email{jpbell@uwaterloo.ca}
\email{cfrwu@uwaterloo.ca}
\email{k29wu@uwaterloo.ca}

\dedicatory{Dedicated to Don Passman on the occasion of his 75th birthday.}



\begin{abstract} 
An algebra $A$ satisfies the Dixmier-Moeglin equivalence if we have the equivalences: $$P~{\rm primitive}\iff P~{\rm rational}\iff P ~{\rm locally~closed~}\qquad~{\rm for}~P\in {\rm Spec}(A).$$ We study the robustness of the Dixmier-Moeglin equivalence under extension of scalars and under the formation of Ore extensions.  In particular, we show that the Dixmier-Moeglin equivalence is preserved under base change for finitely generated complex noetherian algebras. We also study Ore extensions of finitely generated complex noetherian algebras $A$. If $T:A\to A$ is either a $\mathbb{C}$-algebra automorphism or a $\mathbb{C}$-linear derivation of $A$, we say that $T$ is \emph{frame-preserving} if there exists a finite-dimensional subspace $V\subseteq A$ that generates $A$ as an algebra such that $T(V)\subseteq V$.  We show that if $A$ is of finite Gelfand-Kirillov dimension and has the property that all prime ideals of $A$ are completely prime and $A$ satisfies the Dixmier-Moeglin equivalence then the Ore extension $A[x;T]$ satisfies the Dixmier-Moeglin equivalence whenever $T$ is a frame-preserving derivation or automorphism. 
\end{abstract}


\maketitle

%
%
%
%
\section{Introduction} 
The Dixmier-Moeglin equivalence is an important result in the representation theory of enveloping algebras of finite-dimensional Lie algebras that has since been extended to many other settings.   At its core, the motivation for this equivalence lies in trying to understand the irreducible representations of an algebra---this is often a very difficult problem.  As a means to deal with such thorny representation-theoretic issues, Dixmier proposed that one should instead find a ``coarser'' understanding by finding the annihilators of simple modules---these are the so-called primitive ideals and they form a subset of the prime spectrum.  This approach has been used successfully in the study of enveloping algebras and their representations. A concrete characterization of the primitive ideals of enveloping algebras of finite-dimensional complex Lie algebras was given by Dixmier \cite{Dixmier} and Moeglin \cite{Moeglin}.  In this case, we have the following equivalence.
\begin{theorem} Let $\mathcal{L}$ be a finite-dimensional complex Lie algebra and let $P$ be a prime in ${\rm Spec}(U(\mathcal{L}))$.  Then the following are equivalent:
\begin{enumerate}
\item[$(1)$] $P$ is primitive;
\item[$(2)$] $\{P\}$ is locally closed in ${\rm Spec}(R)$;
\item[$(3)$] $P$ is rational.
\end{enumerate}
\end{theorem}
We recall, for the reader's benefit, that a subset $V$ of a topological space $Z$ is \emph{locally closed} if we have $V=C_1\setminus C_2$  for some closed subsets $C_1$ and $C_2$ of $Z$; as stated earlier, a prime ideal $P$ in a ring $R$ is (left) \emph{primitive} if $P$ is the annihilator of some simple left $R$-module.   Finally, a prime ideal $P$ of a noetherian $k$-algebra $R$ is \emph{rational} if the centre of the Artinian ring of quotients of $A/P$, which we denote $Q(A/P)$, is an algebraic extension of the base field.  We let $Z(Q(A/P))$ denote the centre of this ring of quotients and we call it the \emph{extended centre} of $A/P$.
 We point out that the Dixmier-Moeglin equivalence gives both a topological and a purely algebraic characterization of the primitive ideals of an algebra. 

A noetherian algebra for which properties $(1)$--$(3)$ are equivalent for all prime ideals of the algebra is said to satisfy the \emph{Dixmier-Moeglin equivalence}.  More precisely, we have the following definition.  

\begin{definition} Let $R$ be a noetherian algebra.  We say that $R$ satisfies the \emph{Dixmier-Moeglin equivalence} if for every $P\in {\rm Spec}(R)$, the following properties are equivalent:
\begin{enumerate}
\item[($A$)] $P$ is primitive;
\item[($B$)] $P$ is rational;
\item[($C$)] $\{P\}$ is locally closed in ${\rm Spec}(R)$.
\end{enumerate}
\label{definition: DM}
\end{definition}

The Dixmier-Moeglin equivalence has been shown to hold for many classes of algebras, including many quantized coordinate rings and quantized enveloping algebras, algebras that satisfy a polynomial identity, and many algebras that come from noncommutative projective geometry \cite{GL, BG, Procesi, BRS}.  In general, there are many examples of algebras for which the Dixmier-Moeglin equivalence does not  hold.  For example, Lorenz \cite{Lor1} constructed a polycyclic-by-finite group whose complex group algebra is primitive but with the property that $(0)$ is not locally closed in the prime spectrum.  While the Dixmier-Moeglin equivalence need not hold in general, one always has the implications $(C)\implies (A)\implies (B)$ for algebras satisfying the Nullstellensatz (we recall that a $k$-algebra $A$ satisfies the \emph{Nullstellensatz} if every prime ideal is the intersection of the primitive ideals above it and the endomorphism ring of every simple module is algebraic over the base field); in particular, this holds for countably generated noetherian algebras over an uncountable base field (cf. \cite[II.7.16]{BG} and \cite[Prop. 6]{Lor2}).

The focus of this paper is to study the robustness of the Dixmier-Moeglin equivalence under extension of scalars and under Ore extensions.  We recall for the reader's benefit that if $k$ is a field and $A$ is a $k$-algebra then an \emph{Ore extension} $A[x;\sigma,\delta]$ is formed by taking a $k$-linear automorphism $\sigma$ of $A$ and a $k$-linear map $\delta:A \to A$, called a $\sigma$-derivation, which satisfies
$\delta(ab)=\sigma(a) \delta(b)+\delta(a)b$, and then forming an algebra by taking the polynomial ring $A[x]$ and giving it a new multiplication in which multiplication of elements of $A$ is just the ordinary product in $A$ and multiplication by $x$ is given by 
$$x a = \sigma(a) x + \delta(a)\qquad {\rm for~}a\in A.$$

In the case when $\sigma$ is the identity, then $\delta$ is just a derivation of the algebra and we suppress $\sigma$ and write $A[x;\delta]$; in the case when $\delta$ is zero, we suppress $\delta$ and write $A[x;\sigma]$.  

In general, it is known that the Dixmier-Moeglin equivalence is not necessarily preserved under typical extensions used in ring theory. For Ore extensions, an example of Lorenz gives a Laurent polynomial ring $\mathbb{C}[x^{\pm 1},y^{\pm 1}]$ with an automorphism $\sigma$ such that the algebra
$\mathbb{C}[x^{\pm 1},y^{\pm 1}][t,t^{-1};\sigma]$ does not satisfy the Dixmier-Moeglin equivalence (Lorenz' example is stated in the language of group algebras of polycyclic groups). On the other hand, there is a recent example of a finitely generated commutative complex domain with a derivation $\delta$ such that $A[x;\delta]$ does not satisfy the Dixmier-Moeglin equivalence \cite{BLLM}.  
Our main result for extension of scalars shows that for uncountable algebraically closed base fields of characteristic zero, the Dixmier-Moeglin equivalence is preserved under extension of scalars for finitely generated algebras (see Theorem \ref{theorem: extension} for a slightly more general statement). 
\begin{theorem}
Let $A$ be a finitely generated complex noetherian algebra and suppose that $A$ satisfies the Dixmier-Moeglin equivalence.  Then $A\otimes_{\mathbb{C}} F$ satisfies the Dixmier-Moeglin equivalence for all extensions $F$ of $\mathbb{C}$.
\label{theorem: main}
\end{theorem}
One of the nice consequences of this result is that it shows that if $A$ is a finitely generated complex noetherian algebra that satisfies the Dixmier-Moeglin equivalence then, up to finitely many exceptions, the height one primes of $A$ are parametrized by the extended centre in some natural sense.  This agrees with the case of many quantized coordinate rings, where Goodearl and Letzter showed in this setting that after inverting a normal element (i.e., removing a finite set of height one primes and the primes that contain them), the prime spectrum of the resulting algebra is homeomorphic to the prime spectrum of its centre, with the homeomorphisms coming from expansion and contraction of prime ideals.  More precisely, we have the following result.
\begin{corollary} Let $A$ be a finitely generated complex noetherian algebra and suppose that $A$ satisfies the Dixmier-Moeglin equivalence.  If $P$ is a prime ideal of $A$ and $B=A/P$ then $ZB$ has finitely many height one primes, where $ZB$ is the subalgebra of $Q(B)$ generated by $B$ and the centre, $Z$, of $Q(B)$.  
\label{corollary: main}
\end{corollary}
 The second focus of this paper is the study of the Dixmier-Moeglin equivalence under Ore extensions.  Here we use Theorem \ref{theorem: main} and Corollary \ref{corollary: main} to obtain our results.  
As mentioned earlier, there are examples of algebras that satisfy the Dixmier-Moeglin equivalence and 
that have Ore extensions that do not. In practice, however, the Dixmier-Moeglin equivalence is generally preserved under the process of taking Ore 
extensions that arise naturally. For example, the work of Goodearl and Letzter \cite{GL} shows that the class of 
iterated Ore extensions now known as CGL extensions all satisfy the Dixmier-Moeglin equivalence.
  
When one compares the class of CGL extensions with the pathological examples appearing in \cite{Lor1, 
BLLM}, an immediate observation one makes is that in the pathological examples above, the 
automorphism and derivation do not preserve a frame of the algebra.  We recall that in a finitely generated 
$k$-algebra $A$, a \emph{frame} is simply a finite-dimensional subspace $V$ of $A$ that contains $1$ 
and that generates $A$ as a $k$-algebra.  Most Ore extensions that arise in the setting of quantum groups 
and enveloping algebras are \emph{frame-preserving}; that is there is frame of the algebra which is 
mapped into itself by the maps involved in the definition of the Ore extension. We make this precise.
\begin{definition} Let $k$ be a field, let $A$ be a finitely generated $k$-algebra, and let $\sigma:A\to A$ and $\delta: A\to A$ be respectively a $k$-linear automorphism and a $k$-linear $\sigma$-derivation. We say that an Ore extension $A[x;\sigma,\delta]$ is \emph{frame-preserving} if there is a frame $V$ of $A$ such that 
$\sigma(V)\subseteq V$ and $\delta(V)\subseteq V$.
\end{definition}
Frame-preserving Ore extensions are very common; indeed, many of the ``naturally arising'' CGL extensions one encounters in quantum groups are formed via iterated frame-preserving Ore extensions.  Frame-preserving Ore extensions also have the benefit of behaving well with respect to Gelfand-Kirillov dimension---see work of Zhang \cite{Zhang} in particular. 

Our main result on Ore extensions is the following.
\begin{theorem} Let $R$ be a finitely generated noetherian $\mathbb{C}$-algebra of finite Gelfand-Kirillov dimension satisfying the Dixmier-Moeglin equivalence and suppose that all prime ideals of $R$ are completely prime.  Then $R[x;T]$ satisfies the Dixmier-Moeglin equivalence whenever $T$ is either a frame-preserving automorphism or a frame-preserving derivation of $R$.
\label{theorem: main2}
\end{theorem}

The proof of Theorem \ref{theorem: main} uses a curious fact, of which we were previously unaware.  This is that if $k$ is an uncountable algebraically closed field and $A$ is a $k$-algebra whose dimension is strictly less than the cardinality of $k$ then $A\otimes_k F$ is isomorphic to $A$ as a ring with an isomorphism that maps $k$ isomorphically to $F$ whenever $F$ is an algebraically closed extension of $k$ of the same cardinality as $k$.  Using this fact and descent techniques from Irving and Small \cite{IS} (see also Rowen \cite[Theorem 8.4.27]{Rowen}) we can prove that if $A$ is a countable-dimensional complex noetherian algebra and it satisfies the Dixmier-Moeglin equivalence then extending the base field to any countably generated field extension of the complex numbers results in an algebra that still satisfies the Dixmier-Moeglin equivalence.  This is then enough to deduce the more general result about arbitrary extensions.  Theorem \ref{theorem: main} (and the related Corollary \ref{corollary: main}) is then applied in the proof of Theorem \ref{theorem: main2}, which is obtained via a careful study of the behaviour of the extended centre of algebras under frame-preserving maps.

The outline of this paper is as follows. In \S 2, we prove results on base change, including Theorem \ref{theorem: main} and Corollary \ref{corollary: main}. In \S 3, we prove general decomposition results concerning the extended centres of rings having a frame-preserving automorphism or derivation.  Then if \S 4, we prove Theorem \ref{theorem: main2}.

\section{The Dixmier-Moeglin equivalence under base change}
A natural question to ask is whether a $k$-algebra $A$ satisfying the Dixmier-Moeglin equivalence has the property that $A\otimes_k F$ also satisfies the Dixmier-Moeglin equivalence for an extension $F$ of $k$.  

The following lemma is presumably already known by someone, but we are unaware of a published proof.
\begin{lemma} Let $k$ be an uncountable algebraically closed field, let $A$ be a $k$-algebra and let $F$ be an algebraically closed extension of $k$.  If ${\rm dim}_k(A)<|k|$ and $|k|=|F|$ then
$A\otimes_k F\cong A$ as rings and the isomorphism maps $k\subseteq A$ bijectively onto $F=k\otimes_k F\subseteq A\otimes_k F$.
\label{lem: fields}
\end{lemma}
\begin{proof}
Let $\mathcal{B} $ be a $k$-basis for $A$.  Then for each $a,b\in \mathcal{B}$ we may write $ab = \sum_{c\in \mathcal{B}} \lambda_{a,b,c} c$ for some constants $\lambda_{a,b,c}\in k$ with $\lambda_{a,b,c}=0$ for almost all $c$.  We let $k_0$ denote the subfield of $k$ generated over the prime field by
$\lambda_{a,b,c}$ with $a,b,c\in \mathcal{B}$.  Then $k_0$ has cardinality at most $\max(\aleph_0,{\rm dim}_k(A))$ and hence has cardinality strictly less than $k$.   We let $A_0$ denote the $k_0$-vector space spanned by $\mathcal{B} $.  Then by construction $A_0$ is a $k_0$-algebra and $A\cong A_0\otimes_{k_0} k$.  
Then $A\otimes_k F\cong (A_0\otimes_{k_0} k)\otimes_k F \cong A_0\otimes_{k_0} F$.  We claim that $k$ and $F$ have have transcendence bases over $k_0$ of the same size.  To see this, let $X$ and $Y$ be respectively transcendence bases for $k/k_0$ and for $F/k_0$.  Then since $k$ is algebraic over $k_0(X)$ and $k_0(X)$ is infinite, we see that
$|k_0(X)|=|k|$.  Similarly, $|k_0(Y)|=|F|$ and since $|k|=|F|$, we see that $k_0(X)$ and $k_0(Y)$ have the same cardinality.  
We now claim that $|X|=|Y|=|k|$.  To see this, we observe that $X\subseteq k$ and $Y\subseteq F$ giving us $|X|,|Y|\le |k|$.  To see the reverse inequality, we note that $k_0(X)$ is the union over all subfields of the form $k_0(X_0)$ where $X_0$ ranges over all finite subsets of $X$.  For a finitely generated extension $k_0(X_0)$ of $k_0$ we have $k_0(X_0)$ has cardinality at most $\max(\aleph_0,{\rm dim}_k(A))$, since $k_0$ has cardinality bounded above by this quantity.  Thus $k_0(X)$ has cardinality at most $\max(\aleph_0,{\rm dim}_k(A))\cdot |X|$, which is equal to $|X|$, since $|k_0(X)|>\max(\aleph_0,{\rm dim}_k(A))$ and so $|X|>\max(\aleph_0,{\rm dim}_k(A))$.  Thus we see that $|k|=|k_0(X)|=|X|$.  Similarly, $|Y|=|F|=|k|$ and so we see that $k$ and $F$ have transcendence bases of the same size.

Thus we get a $k_0$-algebra isomorphism between $k_0(X)$ and $k_0(Y)$ induced by a bijection of sets from $X$ to $Y$. By uniqueness of algebraic closure, this lifts to a $k_0$-algebra isomorphism between $k$ and $F$. 
(If the reader is concerned about details, by Zorn's lemma one can show that there is an isomorphism $\psi$ from $k$ to $F$ that is the identity on $k_0$ \cite[Theorem 2.6.7]{Scho}.)  Since $\psi$ is the identity on $k_0$, the universal property of tensor products gives that we have a ring homomorphism $A\cong A_0\otimes_{k_0} k \to A\otimes_k F= A_0\otimes_{k_0} F$ given by $a\otimes \lambda \mapsto a\otimes \psi(\lambda)$ for $a\in A_0$ and $\lambda\in k$. The universal property of tensor products also gives that this is an isomorphism of rings and the result follows.
\end{proof}  
We note that although this observation is straightforward, it has many non-trivial consequences.  For example, the first-named author proved that if $k$ is an uncountable algebraically closed field and $A$ is a countably generated (and hence of dimension $<|k|$ as a $k$-vector space) left noetherian $k$-algebra then $A\otimes_k F$ is left noetherian \cite[Theorem 1.2]{Bell} for any extension $F$ of $k$.  We can give a short proof of this result using the preceding lemma.
\begin{theorem} \label{theorem: Bell} Let $k$ be an uncountable algebraically closed field and let $A$ be a left noetherian $k$-algebra with ${\rm dim}_k(A)<|k|$.  Then $A\otimes_k F$ is noetherian for every extension $F$ of $k$.  In particular, if $A$ is also prime then any subfield $L$ of $Q(A)$ containing $k$ is finitely generated as an extension of $k$.
\end{theorem}
\begin{proof} Suppose that there is some field extension $F$ of $k$ such that $A\otimes_k F$ is not left noetherian. Then there is some ideal left ideal $I$ of $A\otimes_k F$ that is countably generated but is not finitely generated.  Let 
$r_1,r_2,\ldots $ be generators for this ideal.  Then each $r_i$ has an expression $\sum a_{i,j}\otimes \lambda_{i,j}$.  We let $F_0$ denote the subfield of $F$ generated generated by $k$ and the $\lambda_{i,j}$.  Then by construction $F_0$ is countably generated over $k$ and so $F_0$ and $k$ have the same cardinality since $k$ is uncountable.  Moreover, $A\otimes_k F_0$ is not left noetherian by construction.  Now let $\bar{F}_0$ denote the algebraic closure of $F_0$.  Then $A\otimes_k \bar{F}_0$ is a free right $A\otimes_k F_0$-module and so the map $L\mapsto (A\otimes_k \bar{F}_0)L$ gives an inclusion preserving embedding of the poset of left ideals of $A\otimes_k F_0$ into the poset of left ideals of $A\otimes_k \bar{F}_0$.  But $\bar{F}_0$ and $k$ have the same cardinality and so $A\otimes_k \bar{F}_0\cong A$ as rings by Lemma \ref{lem: fields}, contradicting the fact that $A$ is left noetherian.  

Next suppose that $A$ is also prime. We note that if $L$ is a subfield of $Q(A)$ containing $k$ and $L/k$ is not finitely generated then a result of Vamos \cite{Vamos} gives that 
$L\otimes_k L$ is not noetherian.  Since $Q(A)$ is free over $L$ as a right vector space, we again have that $Q(A)\otimes_k L$ is not left noetherian and since $Q(A)\otimes_k L$ is a localization of $A\otimes_k L$, we see that $A\otimes_k L$ is not left noetherian, contradicting the preservation of the noetherian property under base change that we have just demonstrated.
\end{proof}
As a corollary, we can prove the following result about the Dixmier-Moeglin equivalence and base change.  We note that the properties of being locally closed and primitive are isomorphism invariants and so Lemma \ref{lem: fields} can be applied with no problems. The property of rationality really depends upon a base field and so the one subtlety is that we must use that the isomorphism given in the statement of Lemma \ref{lem: fields} restricts to an isomorphism of the corresponding base fields.
\begin{theorem} Let $k$ be an uncountable algebraically closed field of characteristic zero and let $A$ be a countable-dimensional noetherian $k$-algebra satisfying the Dixmier-Moeglin equivalence and let $F$ be an extension of $k$.  Then $A\otimes_k F$ satisfies the Dixmier-Moeglin equivalence.
\label{theorem: extension}
\end{theorem}

\begin{proof} Suppose that there is some field extension $F$ of $k$ for which $A\otimes_k F$ does not satisfy the Dixmier-Moeglin equivalence (where we regard $A\otimes_k F$ as an $F$-algebra when considering rationality of prime ideals).  We shall first show that if there is a field extension $F$ of $k$ such that $A\otimes_k F$ does not satisfy the Dixmier-Moeglin equivalence then there must exist such an $F$ with $|F|=|k|$. We have that $A\otimes_k F$ satisfies the Nullstellensatz (cf. \cite[II.7.16]{BG} and \cite[Prop. 6]{Lor2}) and so we then have that the implication $(B)\implies (C)$ in Definition \ref{definition: DM} does not hold. Then there is a rational prime ideal $P$ of $A\otimes_k F$ (i.e., the extended centre of $(A\otimes_k F)/P$ is an algebraic extension of $F$) that is not locally closed in ${\rm Spec}(A\otimes_k F)$.  In particular, there is an infinite set of prime ideals $Q_1,Q_2,\ldots $ that are height one over $P$.  Since ${\rm dim}_k(A)={\rm dim}_F(A\otimes_k F)$ and $A$ is at most countably infinite-dimensional over $k$, we have
that each of $P, Q_1,Q_2,\ldots $ can be generated as an ideal by a countable subset of $A\otimes_k F$ (for example, we can choose an $F$-basis for each ideal as our set of generators).  In particular, there is some countable-dimensional $k$-vector subspace $W$ of $F$ such that our countable generating sets for $P, Q_1,\ldots $ are all contained in $A\otimes_k W$.  Letting $F_0$ denote the extension of $k$ generated by $W$, we then see that $F_0$ is an extension of $k$ that is generated by a set of cardinality at most $\aleph_0$ and has the property that $P,Q_1,\ldots $ all contract to prime ideals of $A\otimes_k F_0$.  

In particular, if $P_0$ is the contraction of $P$ (that is, $P_0=P\cap (A\otimes_k F_0)$) then $P_0$ is not locally closed in $A\otimes_k F_0$.  But it is straightforward to see that $P_0$ is necessarily rational in ${\rm Spec}(A\otimes_k F_0)$.  Thus if the conclusion to the theorem does not hold, then there exists an extension $F$ of $k$ of the same cardinality of $k$ such that $A\otimes_k F$ does not satisfy the Dixmier-Moeglin equivalence. By Lemma \ref{lem: fields} we have that $A\cong A\otimes_k \bar{F}$ under an isomorphism that sends $k$ to $\bar{F}$ and hence $A\otimes_k \bar{F}$ satisfies the Dixmier-Moeglin equivalence, where $\bar{F}$ is the algebraic closure of $F$.  Finally, we note that $A\otimes_k F$ remains noetherian under base change by Theorem \ref{theorem: Bell}, and $(A\otimes_k F)\otimes_F L$ is Jacobson, and satisfies the Nullstellensatz for every extension $L$ of $F$ \cite[II.7.12 and II.7.16]{BG}.  Since $A\otimes_k \bar{F}$ satisfies the Dixmier-Moeglin equivalence, Irving-Small reduction techniques (see Rowen \cite[Theorem 8.4.27]{Rowen} in regards to where the characteristic zero hypothesis is needed) give that $A\otimes_k F$ does too.  In particular, we have shown that $A\otimes_k F$ satisfies the Dixmier-Moeglin equivalence for every countably generated extension $F$ of $k$.  But the above remarks then show that the Dixmier-Moeglin equivalence holds for $A\otimes_k F$ for every extension $F$ of $k$. The result follows.  
\end{proof}
The following result can be thought of as a strengthening of the locally closed condition.  In particular, it shows that after adjoining the extended centre to a prime algebra satisfying the Dixmier-Moeglin equivalence the resulting algebra has only finitely many height one primes.  This agrees with what has been found in some other settings where the Dixmier-Moeglin equivalence has been proved, in which it is shown that after inverting a single normal element, every height one prime intersects the centre non-trivially.  The following theorem can be seen as a result in a similar vein.  
\begin{theorem} Let $k$ be an uncountable algebraically closed field of characteristic $0$ and let $A$ be a noetherian countable-dimensional $k$-algebra satisfying the Dixmier-Moeglin equivalence.  Then if $P$ is a prime ideal of $A$ and $B=A/P$ then $Z(Q(B))\cdot B\subseteq Q(B)$ has only finitely many height one prime ideals.  
\label{theorem: centre}
\end{theorem}
\begin{proof}
Let $Z=Z(Q(B))$.  By Theorem \ref{theorem: extension} the ring $B\otimes_k Z$ is a $Z$-algebra satisfying the Dixmier-Moeglin equivalence.  Then we have a map $B\otimes_k Z\to BZ$ given by $b\otimes z\mapsto bz$.  This map is onto and we let $Q$ denote the kernel of this map.  We claim that $Q$ is a prime ideal of $B\otimes_k Z$; equivalently, we claim that $BZ$ is a prime ring.  To see this, suppose that $BZ$ is not prime.  Then there exist $x,y\in BZ$ such that $xBZy=(0)$.  We write
$x=\sum_{i=1}^d b_i z_i$ with $b_i\in B$ and $z_i\in Z$ and we write $y=\sum_{j=1}^e b_j' z_j'$ with $b_j'\in B$ and $z_j'\in Z$.  We pick a nonzero regular element $c\in B$ such that $z_i c$ and $z_j'c$ are in $B$ for all applicable $i$ and $j$.  By construction $(xc)B(yc)\subseteq xBZyc = (0)$ and since $xc\in B$ and $yc\in B$ and since $B$ is prime, we then see that either $xc=0$ or $yc=0$.  Since $c$ is regular in $B$, it is regular in $Q(B)$ and so we see that either $x=0$ or $y=0$.  Thus $BZ$ is prime and so $Q$ is a prime ideal of $B\otimes_k Z$.  

Since $B\otimes_k Z$ satisfies the Dixmier-Moeglin equivalence (as a $Z$-algebra) and since $Q$ is prime, we see that
$BZ\cong (B\otimes_k Z)/Q$ is a $Z$-algebra that satisfies the Dixmier-Moeglin equivalence.  Since the centre of $ZB$ is equal to $Z$, we see that $Q$ is a rational prime ideal of $B\otimes_k Z$. Thus $Q$ is locally closed and so $BZ$ has only finitely many height one primes.  \end{proof}
\section{Linear operators on rings}
In this section, we study linear operators on rings with a view towards proving our results about Ore extensions. 

  Throughout this section, we make the following assumptions.
\begin{assume*} We assume that $k$ is an uncountable, algebraically closed field, that $R$ is a finitely generated prime noetherian $k$-algebra, and that $T:R\to R$ is a $k$-linear map that is either an automorphism of $R$ or a derivation of $R$ and that there is a finite-dimensional generating subspace $V$ of $R$ that contains $1$ and is $T$-invariant.  
\end{assume*}
We note that in this case the subspaces $V^n$ are $T$-invariant for all $n\ge 1$.  In particular, for each $n$, we may find a basis $\mathcal{B}_n$ for $V^n/V^{n-1}$ such that $T$ is triangular with respect to this basis.  Then by choosing a subset $\mathcal{C}_n\subseteq V^n$ whose image in $V^n/V^{n-1}$ is $\mathcal{B}_n$ and taking the union of these $\mathcal{C}_n$'s in increasing order with $n$, we see that $T$ is \emph{triangularizable} on $R$; that is, there is a basis, $1=v_0,v_1,\ldots $ for $R$ such that $T(v_i)\in kv_0+\cdots +kv_i$ for all $i\ge 0$. Given $c=\sum_{i=0}^m c_i v_i\in R$ with $c_m\neq 0$.  We call $m$ the \emph{height} of $c$.
\begin{rem}
A nonzero $T$-invariant ideal $I$ of $R$ contains a $T$-eigenvector.
\label{rem: rem}
\end{rem}
To see this, we just take a nonzero element $x=\sum_{i=1}^m c_i v_i\in I$ with $c_m\neq 0$ and with $m$ minimal.  Then $T(v_m)-\lambda v_m\in \sum_{i<m} kv_i$ for some $\lambda\in k$ and so 
$T(x)-\lambda x\in I$ is in $\sum_{i<m} kv_i$.  By minimality of $m$ we see that $T(x)=\lambda x$.

We note that the operator $T$ extends to $Q(R)$.  We then let $\mathcal{X}(R)$ to be the collection of subfields $F$ of $Z(Q(R))$ with the property that $F$ is generated as a field extension of $k$ by a finite-dimensional $T$-invariant $k$-vector space.
\begin{equation}
Z_{{\rm fd}}:= \bigcup_{F\in \mathcal{X}(R)} F.
\end{equation}

We remark that $Z_{{\rm fd}}$ is in fact a field, since if $F$ and $F'$ are fields generated by finite-dimensional $T$-invariant subspaces $V$ and $V'$ of $Z(Q(R))$ then $V+V'$ is $T$-invariant and generates a field containing both $V$ and $V'$. Since all subfields of $Z(Q(R))$ are finitely generated (see Theorem \ref{theorem: Bell}), we see that $Z_{{\rm fd}}$ is in fact the field of fractions of an algebra generated by a single finite-dimensional $T$-invariant subspace of $Z(Q(R))$.  By construction, $Z_{{\rm fd}}$ is $T$-invariant.  We use the notation $Z_{{\rm fd}}$ to capture the fact that it is generated by central finite-dimensional $T$-invariant subspaces.
  
Our key lemma is the following result that provides a decomposition of $Z_{{\rm fd}}$.
\begin{lemma}
There exists a family of central subalgebras
$k=C_0\subseteq C_1\subseteq \cdots \subseteq C_m$ of $Z(Q(R))$ such that:
\begin{enumerate}
\item for each $i<m$, there exists $y_{i+1}\in Z(Q(R))$ such that $C_{i+1}=C_i[y_{i+1}]$;
\item for each $i<m$ there is some $\lambda_i\in k$ such that $T(y_i)-\lambda_i y_i \in {\rm Frac}(C_i)$;
\item for each $i\le m$, ${\rm Frac}(C_i)$ is a $T$-invariant subfield of $Z(Q(R))$;  
\item for each $i\le m$ there is a finite-dimensional $T$-invariant subspace of $C_i$ that generates $C_i$ as a $k$-algebra;
\item ${\rm Frac}(C_m)=Z_{{\rm fd}}$.
\end{enumerate}
\label{lem: key}\end{lemma}
\begin{proof} Consider all finite chains of commutative subalgebras of $Z_{{\rm fd}}$ satisfying conditions (1)--(5).  By Theorem \ref{theorem: Bell}, $Z_{{\rm fd}}$ satisfies the ascending chain condition on subextensions of $k$ in $Z_{{\rm fd}}$ and hence among all such chains $C_0\subseteq \cdots \subseteq C_m$, we can pick one with $F:={\rm Frac}(C_m)$ maximal.  We claim that ${\rm Frac}(C_m)=Z_{{\rm fd}}$, which then gives condition (5).  

To see this, suppose that there is some $z\in Z_{{\rm fd}}$ that is not in $F$.  Now $Z_{{\rm fd}}$ is generated as a field by a finite-dimensional $T$-invariant subspace $V$ and so $V\not\subseteq F$.  Thus we may assume that $z\in V$. Consider the ideal $I$ of $R$ consisting of elements $x$ such that $xV\subseteq R$.  Then $I$ is a nonzero $T$-invariant ideal since $V$ is finite-dimensional and $T$-invariant and hence $I$ contains a $T$-eigenvector $D$ by Remark \ref{rem: rem}.  Then we can write $z=cD^{-1}$.  Now let $1=v_0,v_1,\ldots $ be a $T$-triangular basis for $R$ and write $c=\sum_{i=0}^t c_i v_i$ with $c_t\neq 0$.  Now we consider all elements of the form $bD^{-1}\in Z_{{\rm fd}}$ with $b\in R$ that are not in $F$.  Among all such elements we pick $y:=bD^{-1}$ such that $b$ has minimal height.  
Then $T(bD^{-1}) = b'D^{-1}$ where the height of $b'$ is at most $t$.  In particular, there is some $\lambda\in k$ such that $b'-\lambda b$ has height strictly less than $t$.  Thus by minimality of $t$ we see that $(b'-\lambda b)D^{-1}\in F$.   In other words, $T(y)-\lambda y \in F$. Define $C_{m+1}:=C_m[y]$.  Then by construction $y\not\in F$ and so the field of fractions of $C_{m+1}$ strictly contains the field of fractions of $C_m$.  Also by construction ${\rm Frac}(C_{m+1})$ is $T$-invariant. 
Since $y$ is contained in the finite-dimensional $T$-invariant space ${\rm Span}(v_0,\ldots ,v_t)D^{-1}$, we see that $y$ is in $Z_{{\rm fd}}$ and that there is a finite-dimensional $T$-invariant subspace of $C_{m+1}$ that generates $C_{m+1}$. This contradicts the maximality of $F$ among all chains satisfying (1)--(4).  This gives the result.
\end{proof}
We need a lemma about eigenvectors in the extensions appearing in Lemma \ref{lem: key}.
\begin{lemma} Let $R$ be a prime noetherian $k$-algebra, let $S=R[x;T]$ in the case when $T$ is a derivation and $S=R[x,x^{-1};T]$ when $T$ is an automorphism, and let $P$ be a prime ideal of $S$ that intersects $R$ trivially.  Then given a decomposition
$k=C_0\subseteq C_1\subseteq \cdots \subseteq C_m$ of $Z_{{\rm fd}}$ as given in Lemma \ref{lem: key}, we have that if $P$ is rational then, given $j$, there is a finite set of monic polynomials $f_1(t),\ldots ,f_r(t)\in {\rm Frac}(C_{j-1})[t]$ such that $f_i(y_j)$ is a $T$-eigenvector for $i=1,\ldots ,r$ and any monic polynomial $g(t)\in  {\rm Frac}(C_{j-1})[t]$ for which $g(y_j)$ is a $T$-eigenvector is in the multiplicative semigroup generated by $k^*$ and $f_1(y_j),\ldots ,f_r(y_j)$.  
\label{lem: finite}
\end{lemma}
\begin{proof}
Let $\mathcal{N}$ denote the set of natural numbers $n$ for which there exists a monic polynomial $f(t)\in {\rm Frac}(C_{j-1})[t]$ of degree $n>0$ that is a $T$-eigenvector.  We have $T(y_j)-\lambda y_j\in {\rm Frac}(C_{j-1})$ and hence if $n\in \mathcal{N}$ and $f(y_j)$ is monic in $y_j$ of degree $n$ and a $T$-eigenvector then $T(f(y_j))-\lambda_n f(y_j)$ is of smaller degree in $y_j$, where $\lambda_n=n\lambda$ if $T$ is a derivation and $\lambda_n=\lambda^n$ in the automorphism case.  Since $T(f(y_j))$ is an eigenvector, we see that the corresponding eigenvalue must be $\lambda_n$.

Since a product of eigenvectors is an eigenvector we see that $\mathcal{N}$ is closed under addition.  In particular, there exist $n_1,\ldots ,n_r\in \mathcal{N}$ such that every element of $\mathbb{N}$ can be written as a nonnegative integer linear combination of $n_1,\ldots ,n_r$.  Let $f_i(t)$ be a monic polynomial in ${\rm Frac}(C_{j-1})[t]$ of degree $n_i$ such that $f_i(y_j)$ is a $T$-eigenvector. If $f(t)$ is a monic polynomial of degree, say $n$, such that in ${\rm Frac}(C_{j-1})[t]$ such that $f(y_j)$ is a $T$-eigenvector then there exist $a_1,\ldots ,a_r$ such that $f(y_j)$ and $g(y_j):=f_1(y_j)^{a_1}\cdots f_r(y_j)^{a_r}$ are monic eigenvectors of degree $n$ and hence both have corresponding eigenvalue $\lambda_n$.  It follows that if we let $z:=f(y_j)g(y_j)^{-1}$ then $T(z)=z$ in the automorphism case and $T(z)=0$ in the derivation case.  In particular, $z$ is in the extended centre of $S/P$.  By the rationality assumption, we then see that there is some $\gamma\in k$ such that $f(y_j)=\gamma g(y_j)$ in $Q(S/P)$ and since $Q(R)$ embeds in $Q(S/P)$ and $f(y_j)$ and $g(y_j)$ are in $Q(R)$, we see that $f(y_j)$ is in the semigroup generated by $k^*$ and $f_1(y_j),\ldots ,f_r(y_j)$ and so the result follows. 

\end{proof}
\begin{lemma}
Let $R$ be a prime noetherian $k$-algebra in which all prime ideals are completely prime, let $S=R[x;T]$ in the case when $T$ is a derivation and $S=R[x,x^{-1};T]$ when $T$ is an automorphism, and let $P$ be a prime ideal of $S$ that intersects $R$ trivially.  Then if $P$ is a rational prime ideal of $S$ then there is a nonzero $T$-eigenvector $c\in R$ such that every nonzero $T$-invariant prime ideal of $R$ contains $c$.
\label{lem: eigenvector}
\end{lemma}
\begin{proof} Let $P$ be a prime ideal of $S$ whose intersection with $R$ is trivial and suppose that $P$ is rational.

By Lemma \ref{lem: key} there is a chain of central subalgebras $k=C_0\subseteq C_1\subseteq \cdots \subseteq C_m$ of $Z_{{\rm fd}}$ satisfying conditions (1)--(4) of the statement of the lemma.  
Then for $i=0,\ldots ,m$ we let $R_i = RC_i \subseteq Q(R)$ and we let $S_i = C_i\setminus \{0\}$.  Then 
$S_m^{-1}R_m$ is just the subalgebra of $Q(R)$ generated by $Z_{{\rm fd}}$ and $R$.

 Now by Theorem \ref{theorem: centre}, $Z(Q(R)) R$ has only finitely many height one prime ideals and so by Remark \ref{rem: rem} there is a nonzero $T$-eigenvector $a$ such that every nonzero ideal of $Z(Q(R)) R$ contains a power of $a$.  In particular, if $Q$ is a $T$-invariant prime ideal of $R$ then there is some $d$, depending upon $Q$, such that
$a^d \in QZ(Q(R))$.  We claim in fact, that we must have $a^d \in Z_{{\rm fd}}Q$.  To see this, we remark that, as done in the beginning of the section, we can find triangular basis $q_1,q_2,\ldots $ for $Q$ with respect to $T$ since $Q$ is $T$-invariant.  If $a^d\not\in Z_{{\rm fd}}Q$, then we pick $x\in Z_{{\rm fd}}Q$ with the property that $r>0$ is minimal among all expressions $u:=a^d -x  = z_1 q_1+\cdots +z_r q_r$ with $z_1,\ldots ,z_r\in Z(Q(R))$.  Since $Z_{{\rm fd}}$ is generated as a field by a finite-dimensional $T$-invariant central $k$-vector subspace $V$ of $Q(R)$, and $a^d$ is an eigenvector, and any finite subset of $Q$ is contained in a finite-dimensional $T$-invariant subspace, we see that there is some $m\ge 1$ and some $z\in V^m$ such that
$za^d - zx$ lies in a finite-dimensional $T$-invariant subspace of $Z_{{\rm fd}}Q$.  Then we have that
$u=zz_1 q_1+\cdots + zz_r q_r$ lies in a finite-dimensional $T$-invariant subspace, and so we see that
there must be some nontrivial $k$-dependence relation of the form $$\sum_{j=0}^{\ell} \alpha_i T^i(u) \in 
Z(Q(R)) q_1+\cdots + Z(Q(R)) q_{r-1},$$ which is easily seen to give a non-trivial dependence between
$T^i(zz_r)$ for $i=0,\ldots ,\ell$.  In particular, $zz_r$ lies in a finite-dimensional central $T$-invariant subspace and since $z$ does as well, we see that $z_r = z^{-1}(zz_r)$ is in $Z_{{\rm fd}}$; but now  replacing $x$ by $x-z_r q_r\in Z_{{\rm fd}}Q$ gives a strictly smaller $r$, contradicting the minimality of $r$.

Thus $a^d \in Z_{{\rm fd}}Q=QS_m^{-1}R_m$.  Now let $j$ be the smallest nonnegative integer for which there exists a nonzero $T$-eigenvector $c\in S_j^{-1}R_m$ such that every height one $T$-invariant prime ideal $Q$ of $R$ has the property that there is some power of $c$ in $QS_j^{-1}R_m$.  We claim that $j=0$.
To see this, suppose towards a contradiction, that $j>0$. In this case, we let $Q$ be a $T$-invariant nonzero prime ideal of $R$.  Consider the ring $B:=S_{j-1}^{-1} R_m=F_{j-1}R_m$, where $F_{\ell}$ is the field of fractions of $C_{\ell}$.  By assumption on $j$, there is some nonzero polynomial $f(t)\in F_{j-1}[t]$ such that $f(y_j) c^d \in BQ$.  We pick $f$ of minimal degree, which we denote by $s$, such that $f(y_j)c^d\in BQ$ for some $d$.  Since elements of $F_{j-1}$ are units in $B$, we may assume that $f$ is monic.  Then since $c$ is a $T$-eigenvector and $B$ and $Q$ are $T$-invariant, applying $T$ gives that
$T(f(y_j)) c^d \in BQ$.  Now by properties (2) and (3) of Lemma \ref{lem: key} we have that $T(f(y_j))$ is a polynomial of degree $s$ for which the coefficient of $y_j^s$ is some $\gamma\in k$, where $\gamma\in k$ is $\lambda^s$ in the automorphism case and $s\lambda$ in the derivation case, where $\lambda\in k$ is such that $T(y_j)-\lambda y_j\in {\rm Frac}(C_{j-1})$. 

Thus $T(f(y_j))=\gamma f(y_j)$ by minimality of $s$.  By Lemma \ref{lem: finite}, there exists $z_1,\ldots ,z_m\in {\rm Frac}(C_{j-1})[y_j]$ such that $f(y_j)=\gamma z_1^{a_1}\cdots z_m^{a_m}$ for some nonnegative integers $a_1,\ldots ,a_m$ and some $\gamma\in k^*$.  (This even applies in the case when $s=0$ since we are assuming that our polynomial is monic.)  

Thus we have $(z_1\cdots z_m)^p c^d \in BQ$ for some $p\ge 1$.  Finally, since $z_1\cdots z_m$ and $c$ commute, we then see that every prime ideal $Q$ of $R$ has the property that $S_j^{-1}R_mQ$ contains a power of $z_1\cdots z_m c$, contradicting the minimality of $j$.  Thus we obtain the desired claim.

Hence we have $j=0$. In this case, $S_0^{-1}R_m=R_m$ and so there is some eigenvector $c\in R_m$ such that every height one prime ideal $Q$ of $R$ has the property that $QR_m$ contains a power of $c$.   Now by condition (4) of Lemma \ref{lem: key} and Remark \ref{rem: rem} there is a nonzero $T$-eigenvector $D$ of $R$ such that $zD\in R$ for all $z$ in a finite generating set for $R_m$ as a $k$-algebra.  Since $Q$ is finitely generated and $R_m$ is central, we see that $c^n D^{\ell} \in Q$ for some $\ell\ge 1$ and some $n$.  Since prime ideals of $R$ are completely prime, we see that every height one $T$-invariant prime ideal of $R$ contains $cD$ and so we are done.
\end{proof}

\section{Proof of Theorem \ref{theorem: main2}}
In this section we use results from the preceding section to prove Theorem \ref{theorem: main2}.  One tool that we will require is a result due to Letzter, which helps one deal with algebras that are free modules of finite rank over algebras satisfying the Dixmier-Moeglin equivalence.  We refer the reader to the book of Krause and Lenagan \cite{KL} for the definition of Gelfand-Kirillov dimension and its related properties---in any case, no facts about properties of Gelfand-Kirillov dimension will be required in this paper other than the result of Letzter below and a result due to Zhang \cite{Zhang}, both of which can be taken as black boxes. 

\begin{theorem} (Letzter) \label{theorem: Letzter} Let $R\subseteq S$ be noetherian algebras and suppose that $S$ is a free $R$-module of finite rank on both sides.  Then the following hold: \begin{enumerate}
\item[(i)] $(A)\implies (B)$ for $P\in {\rm Spec}(R)~ \iff ~(A)\implies (B)$ for $P\in {\rm Spec}(S)$;
\item[(ii)] $(B)\implies (A)$ for $P\in {\rm Spec}(R)~\iff ~(B)\implies (A)$ for $P\in {\rm Spec}(S)$;
\item[(iii)] if $S$ has finite GK dimension then:\\
 $(A)\implies (C)$ for $P\in {\rm Spec}(R)~\iff ~(A)\implies (C)$ for $P\in {\rm Spec}(S)$,
\end{enumerate}
where (A), (B), and (C) are as in Definition \ref{definition: DM}.
\end{theorem}
Also, when proving the Dixmier-Moeglin equivalence for rings of the form $R[x;\sigma]$, it will be useful to work instead with the localization $R[x,x^{-1};\sigma]$.  It is a straightforward exercise to show that if $R$ is a countable-dimensional noetherian algebra over an uncountable field that satisfies the Dixmier-Moeglin equivalence then $R[x;\sigma]$ satisfies the Dixmier-Moeglin equivalence if and only if $R[x,x^{-1};\sigma]$ does.  The reason for this is that if $R[x;\sigma]$ satisfies the Dixmier-Moeglin equivalence, then since we are only inverting a single normal element, we see that the rationality and locally closed properties are unaffected and so the Dixmier-Moeglin equivalence holds for the primes that survive in $R[x,x^{-1};\sigma]$. Conversely if $R[x,x^{-1};\sigma]$ satisfies the Dixmier-Moeglin equivalence, then one has that the equivalences hold in $R[x;\sigma]$ for all prime ideals $P$ that do not contain a power of $x$.  Since $x$ is normal, a prime contains a power of $x$ if and only if it contains $x$ and since $R[x;\sigma]/(x)\cong R$ and $R$ satisfies the Dixmier-Moeglin equivalence, we get the full Dixmier-Moeglin equivalence for $R[x;\sigma]$ when $R[x,x^{-1};\sigma]$ satisfies the Dixmier-Moeglin equivalence.

\begin{proof}[Proof of Theorem \ref{theorem: main2}] Since $R$ is a finitely generated algebra over an uncountable field, we have that $R$ satisfies the Nullstellensatz.  Thus it suffices to prove the implication $(B)\implies (C)$ in Definition \ref{definition: DM}.

We first consider the case where $T=\delta$ is a derivation of $R$. Suppose that $P$ is a rational prime ideal of $R[x;\delta]$ and let $P_0=R\cap P$.  Then $P_0$ is a prime ideal of $R$ in the derivation case (see Goodearl and Warfield \cite[Lemma 3.21]{GW}) and hence completely prime by hypothesis.  Then we may replace $R$ by $R/P_0$ and assume that $R\cap P=(0)$.  Then by Lemma \ref{lem: eigenvector} we have that there is some nonzero $\delta$-eigenvector $c$ such that every nonzero $\delta$-invariant prime ideal of $R$ contains $c$. Now let $Q$ be a prime of $R[x;\delta]$ that properly contains $P$.  If $Q\cap R$ is nonzero then $c\in Q$.  Since there are only finitely many minimal prime ideals in $R[x;\delta]$ above $P+(c)$, we see that to show $P$ is locally closed it is sufficient to show that there are only finitely many $Q\supseteq P$ with ${\rm ht}(Q)={\rm ht}(P)+1$ such that $Q\cap R=(0)$.  We note that if $Q\cap R=(0)$ then $Q$ survives in the localization 
$Q(R)[x;\delta]/\tilde{P}$, where $\tilde{P}$ is the expansion ideal of $P$ in the localization.  Now if $\tilde{P}$ is nonzero, then $\tilde{P}$ contains a monic polynomial in $x$ and so $Q(R)[x;\delta]/\tilde{P}$ is a finitely generated $Q(R)$-module and also a prime ring and hence is a simple ring.  Thus there are no such $Q$ in this case.

On the other hand, if $\tilde{P}=(0)$ then if $Q(R)[x;\delta]$ is not simple then there is some nonzero proper ideal $I$.  We pick monic $f(x)\in I$ of minimal degree.  Then $[f(x),r]\in I$ has degree strictly less than $f(x)$ for all $r\in R$ and hence $[f(x),r]=0$; similarly, $[f(x),x]=0$ and so $f(x)$ is central.  Thus $\tilde{P}$ is not rational and so since $Q(R)[x;\delta]/\tilde{P}$ is a localization of $R[x;\delta]/P$ we see that $P$ is not rational, which is a contradiction.  The result follows.   

The automorphism case is slightly trickier.  Now suppose that $T=\sigma$ is an automorphism of $R$. Then be the remarks immediately preceding the proof, $R[x;\sigma]$ satisfies the Dixmier-Moeglin equivalence if and only if $R[x,x^{-1};\sigma]$ satisfies the Dixmier-Moeglin equivalence, so we work with $R[x,x^{-1};\sigma]$. Now consider the set $\mathcal{S}$ of all prime ideals $P$ of $R$ that are invariant under some power of $\sigma$.  Suppose that $R[x,x^{-1};\sigma]$ does not satisfy the Dixmier-Moeglin equivalence.  Then there exists some maximal element $P$ of $\mathcal{S}$ such that $(R/P)[x^m,x^{-m};\sigma^m]$ does not satisfy the Dixmier-Moeglin equivalence, where $m\ge 1$ is such that $\sigma^m(P)=P$.  Then we let $S=R/P$ and let $t=x^m$ and $\tau=\sigma^m$ and we may assume that $S[t,t^{-1};\tau]$ does not satisfy the Dixmier-Moeglin equivalence.  We then claim by maximality of $P$, we have that $S[t,t^{-1};\tau]/Q$ satisfies the Dixmier-Moeglin equivalence for all prime ideals $Q$ that intersect $S$ non-trivially.  To see this, suppose that $Q_0:=Q\cap S$ is nonzero and that $S[t,t^{-1};\tau]/Q$ does not satisfy the Dixmier-Moeglin equivalence. Then $Q_0$ is a semiprime ideal and $S[t,t^{-1};\tau]/Q$ is a homomorphic image of $(S/Q_0)[t,t^{-1};\tau]$ and hence $(S/Q_0)[t,t^{-1};\tau]$ does not satisfy the Dixmier-Moeglin equivalence. Since $S$ is noetherian, $Q_0$ is the intersection of a finite set of minimal primes $L_1,\ldots ,L_t$ above it.  Then $\tau$ permutes these primes, and so there is some $d$ such that $\tau^d$ fixes $L_1,\ldots ,L_t$. Then we claim that Letzter's theorem (Theorem \ref{theorem: Letzter}) gives that
$(S/Q_0)[t^d,t^{-d};\tau^d]$ does not satisfy the Dixmier-Moeglin equivalence.  To see this, suppose that the Dixmier-Moeglin equivalence holds for $(S/Q_0)[t^d,t^{-d};\tau^d]$.  Then since $S$ has finite GK dimension and $\tau$ is frame-preserving, we see by a result of Zhang \cite{Zhang} that $(S/Q_0)[t,t^{-1};\tau]$ has finite Gelfand-Kirillov dimension. Then since we have $(B)\implies (A)\implies (C)$ for $(S/Q_0)[t^d,t^{-d};\tau^d]$, Letzter's result gives that these implications hold in $(S/Q_0)[t,t^{-1};\tau]$ and in particular, we get that $(B)\implies (C)$ and so the Dixmier-Moeglin equivalence holds, which is a contradiction. 

It is straightforward to see that $(S/Q_0)[t^d,t^{-d};\tau^d]$ is semiprime and in fact the intersection of the expansion ideals of $L_1,\ldots ,L_t$ are zero and that these expansion ideals are the minimal prime ideals of $(S/Q_0)[t^d,t^{-d};\tau^d]$.  Thus if $S/Q_0[t^d,t^{-d};\tau^d]$ does not satisfy the Dixmier-Moeglin equivalence then there is some $i$ such that $(S/L_i)[t^d,t^{-d};\tau^d]$ does not either.  But $L_i$ corresponds to a prime ideal of $R$ that properly contains $P$ and that is invariant under some iterate of $\sigma$.  This contradicts maximality of $P$.  

It follows that it is enough to consider prime ideals $Q$ of $S[t,t^{-1};\tau]$ with $Q\cap S=(0)$.  We only need to show the implication $(B)\implies (C)$ given in Definition \ref{definition: DM} and so we may assume that $Q$ is rational.  Then by Lemma \ref{lem: eigenvector} we have that there is some nonzero element of $S$ such that every prime ideal of $S$ above $Q$ contains this element and so if $Q$ is not locally closed in ${\rm Spec}(S[t,t^{-1};\sigma])$ then we must have an infinite set of prime ideals above it that intersect $S$ trivially.  But this means that these prime ideals survive in the localization $Q(S)[t,t^{-1};\sigma]/\tilde{Q}$, where $\tilde{Q}$ is the expansion ideal of $Q$ in this localization.  But as we argued in the derivation case, we have that this ring is either simple, in which case we are done, or $\tilde{Q}$ is not rational in which case $Q$ is not rational, which is a contradiction.  This gives us the implication $(B)\implies (C)$ and thus gives the desired result.  
\end{proof}
We remark that the complex numbers can be replaced by any algebraically closed uncountable field of characteristic zero in the above proof.

We also note that we have not been able to prove the result for general Ore extensions with an automorphism $\sigma$ and a $\sigma$-derivation $\delta$.  It would be interesting if this could be worked out and if the completely prime hypothesis could be removed.  This would then allow one to prove results about iterated frame-preserving Ore extensions, which would then recover general results concerning the Dixmier-Moeglin equivalence for many important classes of algebras.  



\section*{Acknowledgments} We thank the referee for many helpful comments and suggestions.  We also thank George Bergman and Stefan Catoiu for many helpful comments.

 \end{document}